\documentclass[twoside]{article}

\usepackage[margin=2cm,bottom=2.5cm]{geometry}
\usepackage{amsmath,amsthm,amssymb,latexsym}

\theoremstyle{plain}%
\newtheorem{proposition}{Proposition}[section]%
\newtheorem{theorem}[proposition]{Theorem}%
\newtheorem{lemma}[proposition]{Lemma}%

\newcommand{\mc}[1]{\mathcal{#1}}
\newcommand{\proten}{{\widehat{\otimes}}}
\newcommand{\id}{\operatorname{id}}
\newcommand{\ip}[2]{{\langle {#1} , {#2} \rangle}}
\newcommand{\lin}{\operatorname{lin}}
\newcommand{\vnten}{\overline\otimes}
\newcommand{\G}{{\mathbb G}}
\newcommand{\C}{{\mathbb C}}
\newcommand{\R}{{\mathbb R}}
\newcommand{\Tr}{\operatorname{Tr}}

\begin{document}

\large
\title{A Note on Operator Biprojectivity of Compact Quantum Groups}
\author{Matthew Daws}
\date{}
\maketitle

\begin{abstract}
Given a (reduced) locally compact quantum group $A$, we can consider the
convolution algebra $L^1(A)$ (which can be identified as the predual of the von
Neumann algebra form of $A$).  It is conjectured that $L^1(A)$ is operator
biprojective if and only if $A$ is compact.  The ``only if'' part always holds,
and the ``if'' part holds for Kac algebras.  We show that if the splitting morphism
associated with $L^1(A)$ being biprojective can be chosen to be completely positive,
or just contractive, then we already have a Kac algebra.  We give another proof of
the converse, indicating how modular properties of the Haar state seem to be
important.


\end{abstract}

\section{Introduction}

A Banach algebra $A$ is \emph{biprojective} if the multiplication map
$\Delta_*:A\proten A\rightarrow A$ has a right inverse in the category of
$A$-bimodule maps.  This can be thought of as a ``finiteness condition''.
In particular, the group algebra $L^1(G)$ is biprojective if and only if
$G$ is compact, see \cite[Chapter~IV, Theorem~5.13]{hel}.

When dealing with more non-commutative (or ``quantum'') algebras (here we
focus on $L^1(G)^* = L^\infty(G)$ when we suggest that the classical situation
is commutative) there is a large amount of evidence that \emph{operator spaces}
form the correct category to work in.  For example, if we consider the Fourier
algebra $A(G)$, then $A(G)$ is \emph{operator} biprojective if and only if
$G$ is discrete, \cite{wood}.  When $G$ is abelian, as $A(G) \cong L^1(\hat G)$,
and $\hat G$ is compact if and only if $G$ is discrete, this result is in
full agreement with what we might expect.  By contrast, if we ask when $A(G)$
is biprojective, then, if $G$ is discrete and almost abelian (contains a finite-index
abelian subgroup) then $A(G)$ is biprojective.  Conversely, if $A(G)$ is
biprojective, then $G$ is discrete, and either almost abelian, or is non-amenable
yet does not contain $\mathbb F_2$, see \cite{runde}.

In this note, we shall continue the study of when the convolution algebra
of a (reduced) compact quantum group is operator biprojective.  It was shown in
\cite[Theorem~4.12]{aristov} that if the convolution algebra of a locally compact quantum
group $\G$ is operator biprojective, then $\G$ is already compact.  Conversely,
if $\G$ is a compact Kac algebra, then $\G$ is operator biprojective.
We shall show that if the right inverse to $\Delta_*$ can be chosen to be
completely contractive, then $\G$ must already be a Kac algebra.  We make
some remarks on the general case.  We indicate that the modular theory of
the Haar state seems to be important outside of the Kac case, and it seems
likely that a better understanding of how to deal with how the coproduct
iteracts with the modular automorphism group will be necessary to completely
characterise when the convolution algebra of $\G$ is operator biprojective.

We shall follow the notation of \cite{ER}, and in particular, write $\proten$
for the operator space projective tensor product, and write $\mc{CB}(E,F)$
to denote the space of complete bounded linear maps between operator spaces
$E$ and $F$.

\section{Locally compact quantum groups}

Locally compact quantum groups \cite{kus1, kus3} are an axiomatic framework which
encompass the $L^1(G)$ algebras, the Fourier algebra $A(G)$, and various
``quantum'' examples, for example, Woronowicz's compact quantum groups.  Kac
algebras \cite{ES} are an earlier axiomatic framework which fails to encompass
many of the ``quantum'' examples, for example \cite{woro3}.

However, we shall concentrate on the compact case, which is technically easier.
We shall follow the presentation of \cite{timm},
which in turn closely follows Woronowicz's original papers \cite{woro1} and \cite{woro2}.
See also readable, non-technical accounts in \cite{kus3}, and the survey \cite{maes},
although be aware that these sources use different notation.

A compact quantum semigroup is a unital C$^*$-algebra $A$ equipped with a unital
$*$-homomorphism $\Delta:A\rightarrow A\otimes_{\min} A$ such that $(\Delta\otimes\iota)\Delta
= (\iota\otimes\Delta)\Delta$.  A compact quantum group is a compact quantum
semigroup $(A,\Delta)$ which satisfies the \emph{cancellation laws}, namely that
\[ \Delta(A)(A\otimes 1) := \lin\{ \Delta(a)(b\otimes 1) : a,b\in A \}, \quad
\Delta(A)(1\otimes A), \]
are both dense in $A\otimes_{\min}A$.  If $G$ is a compact semigroup, then we may
set $A=C(G)$ and $\Delta(f)(s,t) = f(st)$ to get a compact quantum semigroup $(A,\Delta)$.
Then the cancellation laws correspond to $G$ having the cancellation laws: namely
that if $st=sr$ for $s,t,r\in G$, then $t=r$, and similarly with the orders reversed.
As sketched in \cite{maes}, these are equivalent to $G$ being a group.

From now on, fix a compact quantum group $(A,\Delta)$.
These axioms imply that $A$ carries a unique \emph{Haar state}, that is, a state
$\varphi\in A^*$ such that
\[ (\varphi\otimes\iota)\Delta(a) = \varphi(a) 1 = (\iota\otimes\varphi)\Delta(a)
\qquad (a\in A). \]
We can form the GNS construction $(H,\Lambda)$ for $\varphi$.  We shall always suppose
that $(A,\Delta)$ is \emph{reduced}, that is, that $\varphi$ is faithful.  As such,
we shall identify $A$ with a concrete C$^*$-algebra acting on $H$.  If $\varphi$ is
not faithful, then we may quotient by its kernal $N=\{ a\in A : \varphi(a^*a)=0 \}$
to obtain a reduced compact quantum group.  Note that $N$ is an ideal because
$\varphi$ is a KMS weight (see below), see the details in \cite[Theorem~2.1]{bmt}.

Let $M = A''$ be the von Neumann algebra generated by $A$.  Then $\Delta$ extends
to a normal $*$-homomorphism $\Delta:M\rightarrow M\vnten M$.  Then, by
\cite[Theorem~7.2.4]{ER}, $(M\vnten M)_* = M_* \proten M_*$ and normality of $\Delta$
induces a complete contraction $\Delta_*:M_* \proten M_* \rightarrow M_*$.  That
$\Delta$ is coassociative implies that $\Delta_*$ is associative, so $M_*$ becomes a
completely contractive Banach algebra.  If we started with a compact group $G$, then
$M_*$ is nothing but $L^1(G)$, and so we refer to $M_*$ as the \emph{convolution algebra}
of $(A,\Delta)$.  For more on (locally) compact quantum groups in the von Neumann
algebra setting see \cite{kus2}.

A \emph{finite-dimensional corepresentation} of
$(A,\Delta)$ is a matrix $u=(u_{i,j}) \in \mathbb M_n(A)$ such that
\[ \Delta(u_{ij}) = \sum_{k=1}^n u_{ik} \otimes u_{kj}
\qquad (1\leq i,j\leq n). \]
There are suitable notions of \emph{intertwiner} between corepresentations,
and what an \emph{irreducible} corepresentation is.  Every
finite-dimensional corepresentation can be written as the direct sum of
irreducible corepresentations.  Using the Haar state, it can be shown that every
finite-dimensional corepresentation is equivalent to a \emph{unitary} one, that is,
where $u\in\mathbb M_n(A)$ is unitary.  The general corepresentation theory of
$(A,\Delta)$ parallels the representation theory of compact groups very closely.

Let $\{u^\alpha = (u^\alpha_{ij})_{i,j=1}^{n_\alpha} : \alpha\in\mathbb A\}$ be
a maximal family of finite-dimensional
irreducible unitary co-representations of $(A,\Delta)$.  Let $\alpha_0\in\mathbb A$
be such that $v^{\alpha_0} = 1$, the trivial corepresentation.  Let $\mc A$ be the
algebra generated by $\{ u^\alpha_{ij} : \alpha\in\mathbb A, 1 \leq i,j \leq n_\alpha \}$
in $A$.  Then $\mc A$ is a \emph{Hopf $*$-algebra}, and
$\{ u^\alpha_{ij} : \alpha\in\mathbb A, 1 \leq i,j \leq n_\alpha \}$ forms a basis for $\mc A$.
This means that $\mc A$ is a $*$-algebra, that $\Delta$ restricts to give a $*$-homomorphism
$\Delta: \mc A\rightarrow \mc A\otimes \mc A$ (the algebraic tensor product) and there
exist maps $\epsilon:\mc A\rightarrow\mathbb C$ and $S:\mc A\rightarrow\mc A$, the
\emph{counit} and \emph{antipode}, satisfying the usual properties.
Indeed, for $\alpha\in\mathbb A$ and $1\leq i,j\leq n_\alpha$, we have that
\begin{gather*}
\Delta\big( u^\alpha_{i,j} \big) = \sum_{k=1}^{n_\alpha} u^\alpha_{i,k}
\otimes u^\alpha_{k,j}, \quad
S(u^\alpha_{i,j}) = \big( u^\alpha_{j,i} \big)^*, \quad
\epsilon\big( u^\alpha_{i,j} \big) = \delta_{ij}, \quad
\varphi\big( u^\alpha_{i,j} \big) = \delta_{\alpha, \alpha_0}. \end{gather*}
Furthermore, for each $\alpha\in\mathbb A$, there exists a unique positive invertible
matrix $F^\alpha \in \mathbb M_{n_\alpha}$ with $\Tr F^\alpha = \Tr (F^\alpha)^{-1}$,
and such that
\[ \varphi\big( (u^\beta_{ij})^* u^\alpha_{kl} \big) = \delta_{\alpha\beta}
\delta_{jl} \frac{((F^\alpha)^{-1})_{ki}}{\Tr(F^\alpha)},
\quad \varphi\big( u^\beta_{ij} (u^\alpha_{kl})^* \big) = \delta_{\alpha\beta}
\delta_{ik} \frac{F^\alpha_{lj}}{\Tr(F^\alpha)}. \]
The Hopf $*$-algebra $\mc A$ is norm dense in $A$, and is the unique such dense
Hopf $*$-algebra, see \cite[Appendix~A]{bmt}.

These ``$F$-matricies'' allow us to define characters on $\mc A$.  For $z\in\C$,
define
\[ f_z:\mc A\rightarrow \C, \quad u^\alpha_{ij} \mapsto \big((F^\alpha)^z\big)_{ij}. \]
As $F^\alpha$ is positive, the matrix $(F^\alpha)^z$ makes sense.  Then, for
$w,z\in\C$, define
\[ \rho_{z,w}:\mc A\rightarrow\mc A, \quad
u^\alpha_{ij} \mapsto \sum_{k,l=1}^{n_\alpha} f_w(u^\alpha_{ik}) f_z(u^\alpha_{lj})
u^\alpha_{kl}. \]
Then $\rho_{z,w}$ is an automorphism of $\mc A$ with inverse $\rho_{-z,-w}$, and
if $z$ and $w$ are purely imaginary, then $\rho_{z,w}$ is a $*$-automorphism of $\mc A$.

In particular, set
\[ \sigma_z = \rho_{iz,iz}, \quad \tau_z = \rho_{-iz,iz} \qquad (z\in\C). \]
Then $(\sigma_t)_{t\in\R}$ is the (restriction) of the modular automorphism group for
$\varphi$, and $(\tau_t)_{t\in\R}$ is the (restriction) of the scaling group.
For example, we can calculate that $\varphi(a \sigma_{-i}(b)) = \varphi(ba)$ for
$a,b\in\mc A$, a relation which we expect, as $\varphi$ is KMS for $\sigma$.
See \cite{tak2} for more details on modular theory of weights.

\begin{proposition}\label{spec_coreps}
There exists a maximal family of finite-dimensional irreducible unitary co-representations
of $(A,\Delta)$, say $\{v^\alpha = (v^\alpha_{ij})_{i,j=1}^{n_\alpha} : \alpha\in\mathbb A\}$,
with the property that the associated $F$-matricies are all diagonal, say
$F^\alpha$ has diagonal entries $(\lambda^\alpha_i)_{i=1}^{n_\alpha}$, so that
$\sum_i \lambda^\alpha_i = \sum_i (\lambda^\alpha_i)^{-1} = Tr_\alpha$, say.
\end{proposition}
\begin{proof}
Start with some maximal family $\{u^\alpha = (u^\alpha_{ij})_{i,j=1}^{n_\alpha} : \alpha\in\mathbb A\}$
as before.  As each $F^\alpha$ is positive it can be diagonalised by some unitary matrix
$Q^\alpha \in \mathbb M_{n_\alpha}$.  Let $(\lambda^\alpha_i)_{i=1}^{n_\alpha}$ be the
eigenvalues of $F^\alpha$, so that $\Tr(F^\alpha) = \sum_i \lambda^\alpha_i =
\Tr((F^\alpha)^{-1}) = \sum_i (\lambda^\alpha_i)^{-1}$.
Then $(Q^\alpha)^* F^\alpha Q^\alpha$ is the diagonal matrix with entries
$(\lambda^\alpha_i)_{i=1}^{n_\alpha}$.  Set
\[ v^\alpha_{ij} = \big( (Q^\alpha)^* u^\alpha Q^\alpha \big)_{ij}
= \sum_{k,l=1}^{n_\alpha} \overline{Q^\alpha_{ki}} u^\alpha_{kl} Q^\alpha_{lj}
\qquad (\alpha\in\mathbb A, 1\leq i,j\leq n_\alpha). \]
It is now routine to check that $v^\alpha$ is a unitary corepresentation matrix,
and that the properties above still hold for the family $\{ v^\alpha_{ij} \}$.
For example, we see that
\begin{align*} \varphi\big( (v^\beta_{ij})^* v^\alpha_{kl} \big) &= 
\varphi\Big(\Big( \sum_{r,s} \overline{Q^\beta_{ri}} u^\beta_{rs} Q^\beta_{sj} \Big)^*
\sum_{t,p} \overline{Q^\alpha_{tk}} u^\alpha_{tp} Q^\alpha_{pl} \Big)
= \sum_{r,s,t,p} Q^\beta_{ri} \overline{Q^\beta_{sj}} \overline{Q^\alpha_{tk}} Q^\alpha_{pl}
  \varphi\big( (u^\beta_{rs})^* u^\alpha_{tp} \big) \\
&= \delta_{\alpha\beta} \frac{1}{\Tr(F^\alpha)} \sum_{r,s,t} Q^\beta_{ri} \overline{Q^\beta_{sj}}
  \overline{Q^\alpha_{tk}} Q^\alpha_{sl} ((F^\alpha)^{-1})_{tr} \\
&= \delta_{\alpha\beta} \frac{1}{\Tr_\alpha} \sum_s (Q^\alpha)^*_{js} Q^\alpha_{sl}
  \big( (Q^\alpha)^* (F^\alpha)^{-1} Q^\alpha \big)_{ki}
= \delta_{\alpha\beta} \delta_{jl} \delta_{ki} \frac{1}{\Tr_\alpha} \frac{1}{\lambda^\alpha_i}.
\end{align*}
Similar calculations show that
\[ \varphi\big( v^\beta_{ij} (v^\alpha_{kl})^* \big)
= \delta_{\alpha\beta} \delta_{ik} \delta_{jl} \frac{\lambda^\alpha_j}{\Tr_\alpha}. \]
and also
\[ f_z(v^\alpha_{ij}) = \delta_{ij} (\lambda^\alpha_i)^z, \qquad
\rho_{z,w}(v^\alpha_{ij}) = (\lambda^\alpha_i)^w (\lambda^\alpha_j)^z v^\alpha_{ij}. \]
\end{proof}

\section{Biprojectivity}

Let $(A,\Delta)$ be a reduced compact quantum group, with associated Haar state
$\varphi$, GNS construction $(H,\Lambda)$, von Neumann algebra $M$ and convolution
algebra $M_*$.  We shall study when $M_*$ is operator biprojective, that is,
whether there is a completely bounded right inverse to $\Delta_*:M_*\proten M_*
\rightarrow M_*$ which is also an $M_*$-bimodule homomorphism.
Henceforth, we shall term such a map $\theta_*$ a \emph{splitting morphism}.

See \cite{aristov,aristov2} for further details on the operator space case, and
\cite[Chapter~IV]{hel} or \cite[Section~4.3]{rundebook} for the classical Banach space setting.

\begin{lemma}
$M_*$ is biprojective if and only if there exists a normal completely
bounded map $\theta:M \vnten M \rightarrow M$ with
\[ \theta\Delta = \id, \quad \Delta\theta = (\theta\otimes\id)(\id\otimes\Delta)
= (\id\otimes\theta)(\Delta\otimes\id). \]
\end{lemma}
\begin{proof}
Suppose that such a $\theta$ exists, so as $\theta$ is normal, there exists
$\theta_* : M_* \rightarrow M_*\proten M_*$ with $\Delta_* \theta_* = \id$.
Then, for $\omega,\tau\in M_*$ and $x\in M$,
\begin{align*} \ip{x}{\theta_*(\omega*\tau)}
&= \ip{\theta(x)}{\Delta_*(\omega\otimes\tau)}
= \ip{(\theta\otimes\id)(\id\otimes\Delta)(x)}{\omega\otimes\tau} \\
&= \ip{(\id\otimes\Delta)(x)}{\theta_*(\omega) \otimes \tau}
= \ip{x}{\theta_*(\omega) * \tau}. \end{align*}
Here we write $*$ for both the product in $M_*$, and the bimodule action of $M_*$
on $M_* \proten M_*$.  Similarly, $\theta_*(\omega * \tau) = \omega * \theta_*(\tau)$,
so we see that $\theta_*$ is a $M_*$-bimodule homomorphism.

The converse is simply a case of reversing the argument.
\end{proof}

In the following section, we shall carefully study the structure of normal completely
bounded maps $M\vnten M\rightarrow M$.  From now on, fix such a map
$\theta:M\vnten M\rightarrow M$ and let $\{ (v^\alpha_{ij})_{i,j=1}^{n_\alpha}
: \alpha\in\mathbb A\}$ be as in Proposition~\ref{spec_coreps}.

\begin{proposition}\label{theta_struc}
We have that $\theta\Delta = \id$ and $\Delta\theta = (\theta\otimes\id)(\id\otimes\Delta)
= (\id\otimes\theta)(\Delta\otimes\id)$ if and only if
there exists a family $\{ X^\alpha \in \mathbb M_{n_\alpha} : \alpha\in\mathbb A\}$
such that, for $\alpha,\beta\in\mathbb A$, $1\leq i,j\leq n_\alpha$ and
$1\leq k,l\leq n_\beta$,
\[ \theta\big( v^\alpha_{ij} \otimes v^\beta_{kl} \big)
= \delta_{\alpha\beta} X^\alpha_{jk} v^\alpha_{il}, \qquad
\sum_{r=1}^{n_\alpha} X^\alpha_{rr} = 1. \]
\end{proposition}
\begin{proof}
The ``if'' part follows as $\mc A$ generates $M$ and $\theta$ is normal.

Conversely, let $x\in M$ and $\alpha\in\mathbb A$.  For $1\leq i,j\leq n_\alpha$,
\begin{align*} \Delta\theta\big( x \otimes v^\alpha_{ij} \big)
&= (\theta\otimes\id)(\id\otimes\Delta)\big( x \otimes v^\alpha_{ij} \big)
= \sum_{r=1}^{n_\alpha} \theta\big( x \otimes v^\alpha_{ir} \big)
\otimes v^\alpha_{rj}.
\end{align*}
Let $a_{ij} = \theta(x \otimes v^\alpha_{ij})$, so that
$\Delta(a_{ij}) = \sum_r a_{ir} \otimes v^\alpha_{rj}$.  As $\Delta$ is
a $*$-homomorphism, for $1\leq k,l\leq n_\alpha$, we have that
\[ \Delta\big(a_{ij} (v^\alpha_{kl})^*\big) = 
\sum_{r,s=1}^{n_\alpha} a_{ir} (v^\alpha_{ks})^* \otimes v^\alpha_{rj}(v^\alpha_{sl})^*. \]
Applying $(\iota\otimes\varphi)$, we see that,
by the calculations in Proposition~\ref{spec_coreps},
\[ \varphi\big(a_{ij} (v^\alpha_{kl})^*\big) 1 =
\sum_{r,s=1}^{n_\alpha} a_{ir} (v^\alpha_{ks})^* \varphi\big(v^\alpha_{rj}(v^\alpha_{sl})^*\big)
= \delta_{jl} \frac{\lambda^\alpha_j}{\Tr_\alpha} \sum_{r=1}^{n_\alpha} a_{ir} (v^\alpha_{kr})^*. \]
As $v^\alpha$ is a unitary matrix, we see that $1 = \sum_k (v^\alpha_{kr})^* v^\alpha_{ks} =
\delta_{rs} 1$ for $1\leq r,s\leq n_\alpha$.  Thus
\[ \sum_{k=1}^{n_\alpha} \varphi\big(a_{ij} (v^\alpha_{kl})^*\big) v^\alpha_{ks}
= \delta_{jl} \frac{\lambda^\alpha_j}{\Tr_\alpha}
\sum_{r,k=1}^{n_\alpha} a_{ir} (v^\alpha_{kr})^* v^\alpha_{ks}
= \delta_{jl} \frac{\lambda^\alpha_j}{\Tr_\alpha} a_{is}. \]
It follows that
\[ a_{is} = \frac{\Tr_\alpha}{\lambda^\alpha_j}
\sum_{k=1}^{n_\alpha} \varphi\big(a_{ij} (v^\alpha_{kj})^*\big) v^\alpha_{ks}
\qquad \big( \alpha\in\mathbb A, 1\leq i,j,s\leq n_\alpha \big). \]

Similarly, if we set $b_{ij} = \theta(v^\alpha_{ij}\otimes x)$, then
$\Delta(b_{ij}) = \sum_r v^\alpha_{ir} \otimes b_{rj}$, and we can show that
\[ b_{sj} = \lambda^\alpha_i \Tr_\alpha \sum_{k=1}^{n_\alpha}
\varphi\big( (v^\alpha_{ik})^* b_{ij} \big) v^\alpha_{sk}
\qquad \big( \alpha\in\mathbb A, 1\leq i,j,s\leq n_\alpha \big). \]

In particular, we see that $\theta(v^\alpha_{ij} \otimes v^\beta_{kl})$ is
in the linear span of $\{ v^\alpha_{is} : 1\leq s\leq n_\alpha\}$, and
the linear span of $\{ v^\beta_{rl} : 1\leq r\leq n_\beta\}$.  Hence
$\theta(v^\alpha_{ij} \otimes v^\beta_{kl}) = 0$ if $\alpha\not=\beta$.
If $\alpha=\beta$, then by linear independence, we see immediately that
\[ \theta(v^\alpha_{ij} \otimes v^\alpha_{kl}) = X^\alpha_{jk} v^\alpha_{il}, \]
for some scalar $X^\alpha_{jk}$.
Finally, as $\sum_k \theta(v^\alpha_{ik}\otimes v^\alpha_{kj})
= v^\alpha_{ij}$, it follows $\sum_k X^\alpha_{kk}=1$, as required.
\end{proof}

\begin{theorem}
Let $(A,\Delta)$ be a compact quantum group with associated von Neumann
algebra $M$.  Let $\theta_*:M_*\rightarrow M_*\proten M_*$ be a splitting
morphism, and suppose further that
$\theta=\theta_*^*$ is an $M$-bimodule map, in the sense that $\theta(\Delta(a)x\Delta(b))
= a \theta(x) b$ for $x\in M\vnten M$ and $a,b\in M$.
Then the Haar state $\varphi$ is tracial, so $(M,\Delta)$ is a Kac algebra.
\end{theorem}
\begin{proof}
%As $\Delta$ is an injective homomorphism, we have that
%\[ \theta(\Delta(a) x \Delta(b)) = a \theta(x) b
%\qquad (a,b\in M, x\in M\vnten M). \]
Let $\alpha\in\mathbb A$ and $1\leq i,j,k\leq n_\alpha$.  As $\theta(x\Delta(b))
= \theta(x)b$ for $x\in M\vnten M$ and $b\in M$, using the
notation of the last proposition, we see that
\begin{equation}\label{eq:one} X^\alpha_{jk} v^\alpha_{ij} (v^\alpha_{ij})^* =
\theta\big( v^\alpha_{ij} \otimes v^\alpha_{kj} \big) (v^\alpha_{ij})^*
= \sum_{l=1}^{n_\alpha} \theta\big( v^\alpha_{ij}(v^\alpha_{il})^* \otimes
v^\alpha_{kj} (v^\alpha_{lj})^* \big). \end{equation}

Now, as $\{ v^\beta_{rs} \}$ forms a basis for the $*$-algebra $\mc A$, and as $\varphi$
picks out the trivial corepresentation $v^{\alpha_0} = 1$, by the calculations
of Proposition~\ref{spec_coreps}, we see that
\[ v^\alpha_{ij}(v^\alpha_{il})^* \otimes
v^\alpha_{kj} (v^\alpha_{lj})^* = \delta_{jl}\frac{\lambda^\alpha_j}{\Tr_\alpha}
1 \otimes \delta_{kl} \frac{\lambda^\alpha_j}{\Tr_\alpha} 1
+ \text{other terms}. \]
By the structure of $\theta$ established in the last proposition,
it follows that
\[ \sum_{l=1}^{n_\alpha} \varphi\theta\big( v^\alpha_{ij}(v^\alpha_{il})^* \otimes
v^\alpha_{kj} (v^\alpha_{lj})^* \big)
= \delta_{jk} \Big( \frac{\lambda^\alpha_j}{\Tr_\alpha} \Big)^2 1
+ \text{other terms}. \]
By applying $\varphi$ to (\ref{eq:one}), we conclude that
\[ X^\alpha_{jk} \frac{\lambda^\alpha_j}{\Tr_\alpha}
= \delta_{jk} \Big( \frac{\lambda^\alpha_j}{\Tr_\alpha} \Big)^2
\quad\text{so that}\quad
X^\alpha_{jk} = \delta_{jk} \frac{\lambda^\alpha_j}{\Tr_\alpha}. \]

We now repeat this argument on the right, so we find that
\begin{align*} X^\alpha_{jk} (v^\alpha_{ij})^* v^\alpha_{ij} &=
(v^\alpha_{ij})^* \theta\big( v^\alpha_{ij} \otimes v^\alpha_{kj} \big)
= \sum_s \theta\big( (v^\alpha_{is})^* v^\alpha_{ij} \otimes (v^\alpha_{sj})^* v^\alpha_{kj} \big) \\
&= \sum_s \delta_{sj} \frac{1}{\lambda^\alpha_i \Tr_\alpha} 
  \delta_{sk} \frac{1}{\lambda^\alpha_k \Tr_\alpha} 1 + \text{other terms}
\end{align*}
Again, by applying $\varphi$ we see that
\[ X^\alpha_{jk} \frac{1}{\lambda^\alpha_i \Tr_\alpha} =
\delta_{jk} \frac{1}{\lambda^\alpha_i \Tr_\alpha} \frac{1}{\lambda^\alpha_k \Tr_\alpha}
\quad\text{so that}\quad
X^\alpha_{jk} = \delta_{jk} \frac{1}{\lambda^\alpha_k \Tr_\alpha}. \]

We hence see that for all $\alpha$ and $1\leq k\leq n_\alpha$, we have
$\lambda^\alpha_k = 1 / \lambda^\alpha_k$.  As $\lambda^\alpha_k > 0$, we see that
$\lambda^\alpha_k = 1$.  In particular, the modular automorphism group $\sigma$ is
trivial, and so $\varphi$ is tracial, as claimed.

Indeed, if $\varphi$ is tracial, then from Proposition~\ref{spec_coreps},
we see that $\lambda^\alpha_j = (\lambda^\alpha_i)^{-1}$ for all $i,j$.  Thus
$\lambda^\alpha_i=1$ for all $i$ and $\alpha$.  It follows that the automorphism
$\rho_{z,w}$ are trivial, and hence also the scaling group is trivial.  So the
antipode $S$ is bounded.  It is now easy to verify the axioms of a compact Kac algebra,
see \cite[Section~6.2]{ES}.
\end{proof}

We note that an argument of Soltan, \cite[Remark~A.2]{soltan}, shows that if a
compact quantum group $(A,\Delta)$ has a faithful family of tracial states (that is,
for non-zero $x\in A$ there is a tracial state $\phi$ with $\phi(x^*x)\not=0$) then
$(M,\Delta)$ is a Kac algebra.

\begin{theorem}
Let $(A,\Delta)$ be a compact quantum group with associated von Neumann
algebra $M$.  Let $\theta_*:M_*\rightarrow M_*\proten M_*$ be a splitting
morphism.  Suppose that $\theta = \theta_*^*$ is completely positive, or
that $\Delta\theta$ is a contraction.  Then $(M,\Delta)$ is a Kac algebra.
\end{theorem}
\begin{proof}
As $\theta(1)=\theta\Delta(1)=1$, if $\theta$ is positive, then $\theta$ is
contractive, so $\Delta\theta$ is contractive.  

We have that $\Delta\theta:M\vnten M\rightarrow M\vnten M$ is contractive, and
is a projection of $M\vnten M$ onto the subalgebra $\Delta(M)$.  A result of
Tomiyama, \cite{tom} or \cite[Theorem~3.4, Chapter~III]{tak1}, tells us that, in
particular, $\Delta\theta(\Delta(a)x\Delta(b)) = \Delta(a) \theta(x) \Delta(b)$
for $a,b\in M$ and $x\in M\vnten M$.  As $\Delta$ is an injective homomorpshim,
the above theorem applies.
\end{proof}

In the following section, we shall show the converse to this corollary: namely
that for a compact Kac algebra $(M,\Delta)$, we can choose $\theta$ to be a
complete contraction; alternatively, see \cite{RX} or \cite{aristov}.

It is shown in \cite{CS} that if we have a completely bounded map
$\theta:M\vnten M\rightarrow M$ with $\theta\Delta=\id$ then there exists
a completely bounded map $\theta_1:M\vnten M\rightarrow M$ which is an $M$-bimodule
map, in the above sense.  However, there is no reason that $\theta_1$ need
be normal, and no reason that the other conditions on $\theta$ will carry over
to $\theta_1$, so that Proposition~\ref{theta_struc} need not apply to $\theta_1$.
We can even choose $\theta_1$ to be completely positive, which were it also
\emph{faithful} would imply, by \cite[Theorem~4.2, Chapter~IX]{tak2}, the existence
of a weight $\omega$ on $M\vnten M$ with interesting modular properties.  Again,
there seems to be no reason to expect that we can choose $\theta_1$ in such a way.

\section{Completely bounded maps}

There is a well-known structure theory for completely bounded maps, \cite[Section~5.3]{ER}.
If $\theta:N\rightarrow (M,H)$ is a completely positive normal map between von Neumann
algebras, then the usual proof of the Stinespring theorem (for example, \cite[Chapter~IV,
Theorem~3.6]{tak1}) can be adapted to show that there exists a Hilbert space $K$,
a \emph{normal} $*$-homomorphism $\pi:N\rightarrow\mc B(K)$ and a bounded map
$U:H\rightarrow K$ such that $\theta(x) = U^* \pi(x) U$ for $x\in N$.

Showing the same for completely bounded maps is not quite as simple, but the details
are worked out in, for example, the proof of \cite[Theorem~2.4]{HM}.  In particular,
given $\theta:N\rightarrow (M,H)$ a completely contractive normal map between von Neumann
algebras, there exist unital completely positive \emph{normal} maps $\phi_1,\phi_2:N
\rightarrow M$ such that
\[ \sigma:\mathbb M_2(N) \rightarrow \mathbb M_2(M); \qquad
\begin{pmatrix} a & b \\ c & d \end{pmatrix} \mapsto
\begin{pmatrix} \phi_1(a) & \theta(b^*)^* \\ \theta(c) & \phi_2(d) \end{pmatrix} \]
is unital completely positive and normal.  One can now follow the presentation in
\cite[Theorem~5.33]{ER} or \cite{Paulsen}, essentially
applying the Stinespring construction
to $\sigma$.  This yields a Hilbert space $K$, a normal $*$-homomorphism $\rho:
\mathbb M_2(N)\rightarrow\mc B(K)$ and an isometry $U:H^2 \rightarrow K$ such that
$\sigma(x) = U^*\rho(x)U$ for $x\in\mathbb M_2(N)$.  Following the proof of
\cite[Theorem~2.3]{HM}, there also exists a normal $*$-homomorphism $\rho':\mathbb M_2(M)'
\rightarrow \rho(\mathbb M_2(N))'$ such that $\rho'(y)U = Uy$ for $y\in\mathbb M_2(M)'$.

Define $\pi:M\vnten M\rightarrow\mc B(K)$, $\pi':M'\rightarrow\mc B(K)$
and $S,T:H\rightarrow K$ by
\[ \pi(x) = \rho \begin{pmatrix} x & 0 \\ 0 & x \end{pmatrix}, \quad
\pi'(y) = \rho' \begin{pmatrix} y & 0 \\ 0 & y \end{pmatrix}, \quad
T(\xi) = \rho\begin{pmatrix} 0 & 0 \\ 1 & 0 \end{pmatrix}
U \begin{pmatrix} \xi \\ 0 \end{pmatrix}, \quad
S(\xi) = U \begin{pmatrix} 0 \\ \xi \end{pmatrix}, \]
for $x\in M\vnten M, y\in M'$ and $\xi\in H$.
So $\pi$ and $\pi'$ are normal $*$-homomorphisms
and $S$ and $T$ are contractions.
Then, for $x\in M\vnten M$ and $\xi,\eta\in H$,
\begin{align*} \big( S^*\pi(x)T\xi \big| \eta \big) &= 
\Big( \rho \begin{pmatrix} x & 0 \\ 0 & x \end{pmatrix}
\rho\begin{pmatrix} 0 & 0 \\ 1 & 0 \end{pmatrix}
U \begin{pmatrix} \xi \\ 0 \end{pmatrix}
\Big| U \begin{pmatrix} 0 \\ \eta \end{pmatrix} \Big) \\
&= \Big( \sigma \begin{pmatrix} 0 & 0 \\ x & 0 \end{pmatrix}
\begin{pmatrix} \xi \\ 0 \end{pmatrix} 
\Big| \begin{pmatrix} 0 \\ \eta \end{pmatrix} \Big)
= \big( \theta(x) \xi \big| \eta \big).
\end{align*}
So $\theta(x) = S^*\pi(x)T$.  Then also, for $y\in M'$ and $\xi\in H$,
\[ T y \xi = \rho\begin{pmatrix} 0 & 0 \\ 1 & 0 \end{pmatrix}
U \begin{pmatrix} y & 0 \\ 0 & y \end{pmatrix}
\begin{pmatrix} \xi \\ 0 \end{pmatrix}
= \rho\begin{pmatrix} 0 & 0 \\ 1 & 0 \end{pmatrix}
\pi'(y) U \begin{pmatrix} \xi \\ 0 \end{pmatrix}
= \pi'(y) T \xi. \]
So $T y = \pi'(y) T$ and similarly $S y = \pi'(y) S$, for $y\in M'$.

Let $M$ be a von Neumann algebra with a normal faithful state $\varphi$,
leading to GNS construction $(H,\Lambda)$ (here we identify $M$ with
a subalgebra of $\mc B(H)$).
We can apply Tomita-Takesaki theory to find an anti-linear isometry $J:H\rightarrow H$
such that $M' = JMJ$ (see \cite{tak2}).  Let $(\sigma_t)_{t\in\R}$ be the modular
automorphism group, and let $\mc A\subseteq M$ be a $*$-subalgebra of elements analytic
for $(\sigma_t)$ such that $\sigma_z(a)\in\mc A$ for $z\in\C$ and $a\in\mc A$.
For $a\in \mc A$, write $a' = J\sigma_{i/2}(a)^* J$.  Then
\[ a' \Lambda(1) = J\sigma_{i/2}(a)^* J \Lambda(1) = \Lambda(a) \qquad (a\in\mc A). \]

\begin{proposition}
Let $M$ be a von Neumann algebra as above, and suppose that $\mc A''=M$.
Let $N$ be a von Neumann algebra.  If $\theta:N\rightarrow M$ is completely bounded
normal map, then we can find a Hilbert space $K$, normal $*$-homomorphisms
$\pi:N\rightarrow\mc B(K)$ and $\pi':M'\rightarrow \pi(N)'$, 
and $\xi_0,\xi_1\in K$ such that the maps
\[ \Lambda(a)\mapsto \pi'(a')\xi_0, \qquad
\Lambda(a)\mapsto \pi'(a')\xi_1 \qquad (a\in\mc A), \]
are bounded, and
\begin{equation}\label{eq:two}
\varphi(\theta(x)a) = \big( \pi(x) \pi'(a') \xi_0 \big| \xi_1 \big)
\qquad (x\in N, a\in\mc A). \end{equation}
Conversely, given such $K,\pi,\pi',\xi_0$ and $\xi_1$, there exists a
completely bounded normal map $\theta:N\rightarrow M$ satisfying (\ref{eq:two}).

Furthermore, $\theta$ is completely positive if and only if we can choose
$\xi_0=\xi_1$.
\end{proposition}
\begin{proof}
As $\mc A''=M$, it follows that $\mc A$ is strongly dense in $M$ and hence
that $\Lambda(\mc A)$ is norm dense in $H$.
If $\theta$ is of the form claimed, then the map $T:\Lambda(\mc A)\rightarrow
K; \Lambda(a)\mapsto \pi'(a')\xi_0$ is bounded and so extends to a bounded linear
map $T:H\rightarrow K$.  Similarly, there exists $S\in\mc B(H,K)$ with
$S\Lambda(a) = \pi'(a')\xi_1$.  Then, for $a,b\in\mc A$ and $x\in N$,
\begin{align*} \big( S^*\pi(x) T \Lambda(a) \big| \Lambda(b) \big)
&= \big( \pi(x) \pi'(a') \xi_0 \big| \pi'(b') \xi_1 \big)
= \big( \pi(x) \pi'((b')^*a') \xi_0 \big| \xi_1 \big) \\
&= \varphi\big( \theta(x) a\sigma_{-i}(b^*) \big),
\end{align*}
as $(b')^* = (J\sigma_{i/2}(b)^*J)^*
= J \sigma_{i/2}(b) J = J \sigma_{i/2}(c)^* J = c'$ if $c = \sigma_{-i}(b^*)$,
and $d\mapsto d'$ is an anti-homomorphism.  By the KMS condition, we see that
\[ \big( S^*\pi(x) T \Lambda(a) \big| \Lambda(b) \big)
= \varphi\big( b^* \theta(x) a \big)
= \big( \theta(x) \Lambda(a) \big| \Lambda(b) \big). \]
  Hence $\theta$ is completely bounded,
as $\theta(x) = S^*\pi(x)T$ for $x\in N$.  If $\xi_0=\xi_1$ then $S=T$ and
$\theta$ is completely positive.

Conversely, given $\theta$, from the discussion above, we can find normal $*$-homomorphisms
$\pi:N\rightarrow \mc B(K)$ and $\pi':M'\rightarrow\pi(N)'$, and bounded maps
$S,T:H\rightarrow K$ with $\theta(x) = S^*\pi(x)T$ for $x\in N$ and
$Sy=\pi'(y)S, Ty=\pi'(y)T$ for $y\in M'$.  Thus, for $x\in N$ and $a\in\mc A$,
\[ \varphi(\theta(x)a) = \big( S^*\pi(x)T\Lambda(a) \big| \Lambda(1) \big)
= \big( S^* \pi(x) \pi'(a') T \Lambda(1) \big| \Lambda(1) \big), \]
so the proof is complete by setting $\xi_0=T\Lambda(1)$ and $\xi_1=S\Lambda(1)$.
If $\theta$ is completely positive, then we can set $S=T$ and hence $\xi_0=\xi_1$.
\end{proof}

Notice that by the KMS condition, the calculations above also show that if
$x,y\in M$ are such that $\varphi(xa) = \varphi(ya)$ for all $a\in\mc A$, then $x=y$.

The following is proved using different methods in \cite{RX} and \cite{aristov}.
Our proof makes explicit how $\varphi$ being tracial, for a Kac algebra, is central
to the proof, and indicates that understanding the modular properties of $\varphi$
for a general compact quantum group will be important in finding a completely
bounded analogue of the following.

\begin{theorem}
Let $(M,\Delta)$ be a compact Kac algebra.  Then there exists a splitting morphism
$\theta_*:M_*\rightarrow M_* \proten M_*$ such that $\theta = \theta_*^*$ is
completely positive.
\end{theorem}
\begin{proof}
We have that $\varphi$ is tracial.  Let $\pi:M\vnten M\rightarrow M\vnten M
\subseteq \mc B(H\otimes H)$ be the trivial representation, let $\xi_0=\xi_1
=\Lambda(1)\otimes\Lambda(1)$, and define $\pi'$ by
\[ \pi'(y) = (J\otimes J) \Delta(JyJ) (J\otimes J)
\qquad (y\in M'). \]
This formula is derived from the natural coproduct on $M'$, see \cite[Section~4]{kus2}.
Let $\mc A$ be the Hopf $*$-algebra associated to $(M,\Delta)$, as before.
Then we can apply the above proposition to see that there exists a completely positive
normal map $\theta:M\vnten M\rightarrow M$ such that
\[ \varphi(\theta(x)a) = \big( x(J\otimes J)\Delta(a^*)(J\otimes J)
\Lambda(1)\otimes\Lambda(1) \big| \Lambda(1)\otimes\Lambda(1) \big), \]
where we use that $\sigma$ is trivial, as $\varphi$ is tracial.  Then
$J\Lambda(a) = \Lambda(a)^*$ for $a\in \mc A$, and so, as $\Delta(a^*)\in\mc A\otimes\mc A$,
\begin{align*} \varphi(\theta(x)a) &=
\big( x(J\otimes J)\Delta(a^*) \Lambda(1)\otimes\Lambda(1) \big| \Lambda(1)\otimes\Lambda(1) \big) \\
&= \big( x(\Lambda\otimes\Lambda)\Delta(a) \big| \Lambda(1)\otimes\Lambda(1) \big)
= (\varphi\otimes\varphi)\big( x\Delta(a) \big).
\end{align*}
In particular,
\[ \varphi(\theta\Delta(x)a) = (\varphi\otimes\varphi)\big( \Delta(xa) \big)
= \varphi(xa), \]
so by the observation above, $\theta\Delta=\id$.  Indeed, one may calculate
(thinking about Proposition~\ref{theta_struc}) that
\[ \theta(v^\alpha_{ij} \otimes v^\alpha_{kl}) = \frac{1}{n_\alpha} \delta_{jk} v^\alpha_{il}, \]
using that $\lambda^\alpha_i=1$ for all $\alpha$ and $i$.  Thus
also $\Delta\theta=(\theta\otimes\id)(\id\otimes\Delta)
=(\id\otimes\theta)(\Delta\otimes\id)$, and so $\theta_*$, the preadjoint
to $\theta$, is a splitting morphism, as required.
\end{proof}

If $\varphi$ is not tracial, then the above proof fails, as for $a\in\mc A$,
\[ \Delta(\sigma_{i/2}(a)^*) = \big((\tau_{i/2}\otimes\sigma_{i/2})\Delta(a)\big)^*, \]
and hence, as $J\Lambda(b) = \Lambda(\sigma_{i/2}(b)^*)$ for $b\in\mc A$,
\[ (J\otimes J)\Delta(\sigma_{i/2}(a)^*)(J\otimes J)(\Lambda(1)\otimes\Lambda(1))
= (\Lambda\otimes\Lambda)\big((\tau_{i/2}\sigma_{-i/2}\otimes\id)\Delta(a)\big). \]
If we continus to form $\theta$ as above, then we find that
\[ \theta\big( v^\alpha_{ij} \otimes v^\beta_{kl}\big) = \delta_{\alpha\beta}
 v^\alpha_{il} \frac{\delta_{jk}}{\Tr_\alpha} \qquad
(\alpha,\beta\in\mathbb A, 1\leq i,j \leq n_\alpha, 1\leq k,l \leq n_\beta ). \]
This is nearly of the correct form, but we find that
\[ \theta\Delta\big( v^\alpha_{ij} \big) = \frac{n_\alpha}{\Tr_\alpha} v^\alpha_{ij}
\qquad (\alpha\in\mathbb A, 1\leq i,j \leq n_\alpha ). \]
Notice that $n_\alpha = \sum_i (\lambda^\alpha_i)^{1/2} (\lambda^\alpha_i)^{-1/2}
\leq \big(\sum_i \lambda^\alpha \big)^{1/2} \big(\sum_i (\lambda^\alpha)^{-1} \big)^{1/2}
= \Tr_\alpha$, it follows that $\theta\Delta=\id$ if and only if $\lambda^\alpha_i=1$
for all $\alpha,i$, that is, again, $\varphi$ is tracial.

\bigskip
\noindent\textbf{Author's address:}
\parbox[t]{5in}{School of Mathematics,\\
University of Leeds,\\
Leeds LS2 9JT\\
United Kingdom}

\smallskip
\noindent\textbf{Email:} \texttt{matt.daws@cantab.net}


\begin{thebibliography}{99}
\normalsize
\newcommand{\bibbook}[3]{\textsc{#1}, \emph{#2}, #3.}
\newcommand{\bibpaper}[6]{\textsc{#1}, \emph{#2}, #3 \textbf{#4} (#5), #6.}
\newcommand{\bibpreprint}[2]{\textsc{#1}, `#2', preprint.}

\bibitem{aristov} \textsc{O.\,Yu. Aristov},
  \emph{Amenability and compact type for {H}opf-von {N}eumann
  algebras from the homological point of view}.
  In \emph{Banach algebras and their applications}, 
  Contemp. Math. \textbf{363}, 15--37.  Amer. Math. Soc., Providence, RI, 2004.

\bibitem{aristov2} \textsc{O.\,Yu. Aristov},
   \emph{Biprojective algebras and operator spaces},
   J. Math. Sci. (New York) \textbf{111} (2002), 3339--3386.

\bibitem{bmt} \textsc{E.~B{\'e}dos}, \textsc{G.~J. Murphy}, and \textsc{L.~Tuset},
   \emph{Co-amenability of compact quantum groups},
   J. Geom. Phys. \textbf{402} (2001), 130--153.

\bibitem{CS} \textsc{E. Christensen and A.\,M. Sinclair},
   \emph{Module mappings into von {N}eumann algebras and injectivity},
   Proc. London Math. Soc. \textbf{71} (1995), 618--640.

\bibitem{ER} \textsc{E.\,G. Effros and Z.-J. Ruan},
   \emph{Operator spaces},
   London Math. Society Monographs, New Series  \textbf{23}.
   Oxford University Press, New York, 2000.

\bibitem{ES} \textsc{M. Enock and J.-M. Schwartz},
   \emph{Kac algebras and duality of locally compact groups},
   Springer-Verlag, Berlin, 1992.

\bibitem{HM} \textsc{U. Haagerup and M. Musat},
   \emph{Classification of hyperfinite factors up to completely bounded isomorphism of their preduals},
   preprint, arXiv:0706.3463 [math.OA]

\bibitem{hel} \textsc{A.\,Ya. Helemskii},
   \emph{The homology of Banach and topological algebras},
   Translated from the Russian by Alan West. Mathematics and its Applications (Soviet Series), \textbf{41}.  Kluwer Academic Publishers Group, Dordrecht, 1989.

\bibitem{kus1} \textsc{J. Kustermans and S. Vaes},
   \emph{Locally compact quantum groups},
   Ann. Sci. \'Ecole Norm. Sup. \textbf{336} (2000), 837--934.

\bibitem{kus2} \textsc{J. Kustermans and S. Vaes},
   \emph{Locally compact quantum groups in the von {N}eumann algebraic setting},
   Math. Scand. \textbf{92} (2003), 68--92.

\bibitem{kus3} \textsc{J. Kustermans},
   \emph{Locally compact quantum groups},
   In \emph{Quantum independent increment processes. {I}} volume 1865 of
   \emph{Lecture Notes in Math.} pages 99--180 (Springer, Berlin, 2005).

\bibitem{maes} \textsc{A. Maes and A. Van Daele},
   \emph{Notes on compact quantum groups},
   Nieuw Arch. Wisk. \textbf{16} (1998), 73--112.

\bibitem{Paulsen} \textsc{V.\, I. Paulsen},
   \emph{Completely bounded maps and dilations},
   Pitman Research Notes in Mathematics Series, \textbf{146}.
   John Wiley \& Sons, Inc., New York, 1986.

\bibitem{RX} \textsc{Z.-J. Ruan and G. Xu},
   \emph{Splitting properties of operator bimodules and operator amenability of Kac algebras}.
   In \emph{Operator theory, operator algebras and related topics} (Timisoara, 1996), 193--216, Theta Found., Bucharest, 1997. 

\bibitem{runde} \textsc{V. Runde},
   \emph{Biflatness and biprojectivity of the Fourier algebra},
   preprint, arXiv:0808.1146 [math.FA]

\bibitem{rundebook} \textsc{V. Runde},
   \emph{Lectures on Amenability},
   Springer-Verlag, Berlin, 2002.

\bibitem{soltan} \textsc{P.\,M. So{\l}tan},
   \emph{Quantum {B}ohr compactification},
  Illinois J. Math. \textbf{49} (2005), 1245--1270.

\bibitem{tak1} \textsc{M. Takesaki},
   \emph{Theory of operator algebras. I.},
   Reprint of the first (1979) edition. Encyclopaedia of Mathematical Sciences, \textbf{124}. Operator Algebras and Non-commutative Geometry, \textbf{5}. Springer-Verlag, Berlin, 2002.

\bibitem{tak2} \textsc{M. Takesaki},
   \emph{Theory of operator algebras. II.},
   Encyclopaedia of Mathematical Sciences, \textbf{125}. Operator Algebras and Non-commutative Geometry, \textbf{6}. Springer-Verlag, Berlin, 2003.

\bibitem{timm} \textsc{T. Timmermann},
   \emph{An invitation to quantum groups and duality.
   From Hopf algebras to multiplicative unitaries and beyond.}
   European Mathematical Society (EMS), Z{\"u}rich, 2008.

\bibitem{tom} \textsc{J. Tomiyama},
  \emph{On the projection of norm one in $W^*$-algebras},
   Proc. Japan Acad. \textbf{33} (1957), 608--612. 

\bibitem{wood} \textsc{P.\,J. Wood},
   \emph{The operator biprojectivity of the {F}ourier algebra},
   Canad. J. Math. \textbf{545} (2002), 1100--1120.

\bibitem{woro1} \textsc{S.~L. Woronowicz},
   \emph{Compact quantum groups}.
   In \emph{Sym\'etries quantiques (Les Houches, 1995)}, 845--884.
   North-Holland, Amsterdam, 1998

\bibitem{woro2} \textsc{S.~L. Woronowicz},
   \emph{Compact matrix pseudogroups},
   Comm. Math. Phys. \textbf{1114} (1987), 613--665.

\bibitem{woro3} \textsc{S.~L. Woronowicz},
   \emph{Twisted {${\rm SU}(2)$} group. {A}n example of a noncommutative differential calculus},
   Publ. Res. Inst. Math. Sci. \textbf{231} (1987), 117--181.

\end{thebibliography}
\end{document}